\documentclass[a4paper,oneside,11pt]{article}

\usepackage{geometry}                % See geometry.pdf to learn the layout options. There are lots.
\usepackage{graphicx}
\usepackage{amsmath,amsfonts,amssymb,amsthm}
\usepackage{epstopdf}
\usepackage{pstricks}
\usepackage[utf8]{inputenc} 
\usepackage[T1]{fontenc} 
 
\usepackage[english]{babel} 

\title{The Poisson boundary of $\text{Out}(F_N)$}
\author{Camille Horbez}

\usepackage[all]{xy}

\usepackage{bbm}

\begin{document}
\maketitle
%\tableofcontents
\newtheorem{de}{Definition} [section]
\newtheorem{theo}[de]{Theorem} 
\newtheorem{prop}[de]{Proposition}
\newtheorem{lemma}[de]{Lemma}
\newtheorem{cor}[de]{Corollary}
\newtheorem{propd}[de]{Proposition-Definition}

\theoremstyle{remark}
\newtheorem{rk}[de]{Remark}
\newtheorem{ex}[de]{Example}
\newtheorem{question}[de]{Question}

\normalsize

\addtolength\topmargin{-.5in}
\addtolength\textheight{1.in}
\addtolength\oddsidemargin{-.045\textwidth}
\addtolength\textwidth{.09\textwidth}

\begin{abstract}
Let $\mu$ be a probability measure on $\text{Out}(F_N)$ with finite first logarithmic moment with respect to the word metric, finite entropy, and whose support generates a nonelementary subgroup of $\text{Out}(F_N)$. We show that almost every sample path of the random walk on $(\text{Out}(F_N),\mu)$, when realized in Culler and Vogtmann's outer space, converges to the simplex of a free, arational tree. We then prove that the space $\mathcal{FI}$ of simplices of free and arational trees, equipped with the hitting measure, is the Poisson boundary of $(\text{Out}(F_N),\mu)$. Using Bestvina-Reynolds' and Hamenstädt's description of the Gromov boundary of the complex $\mathcal{FF}_N$ of free factors of $F_N$, this gives a new proof of the fact, due to Calegari and Maher, that the realization in $\mathcal{FF_N}$ of almost every sample path of the random walk converges to a boundary point. We get in addition that $\partial\mathcal{FF}_N$, equipped with the hitting measure, is the Poisson boundary of $(\text{Out}(F_N),\mu)$.  
\end{abstract}

\section*{Introduction}

Over the past decades, the study of the group $\text{Out}(F_N)$ of outer automorphisms of a finitely generated free group has greatly benefited from the study of its action on some geometric complexes, among which stand Culler and Vogtmann's outer space $CV_N$, which is the space of homothety classes of free, minimal, isometric and simplicial actions of $F_N$ on simplicial metric trees \cite{CV86}, and Hatcher and Vogtmann's complex of free factors \cite{HV98}. A main source of inspiration in this study comes from analogies with mapping class groups of surfaces, and their actions on the associated Teichmüller spaces and curve complexes. We aim at understanding the behaviour of random walks on $\text{Out}(F_N)$-complexes.

Given a countable group $G$ and a probability measure $\mu$ on $G$, the \emph{(right) random walk} on $(G,\mu)$ is the Markov chain on $G$ whose initial distribution is given by the Dirac measure at the identity element, and whose transition probabilities are given by $p(g,h):=\mu(g^{-1}h)$. In other words, the position of the random walk on $(G,\mu)$ at time $n$ is given from its position $g_0=e$ by successive multiplications on the right by independent $\mu$-distributed increments $s_i$, i.e. $g_n=s_1\dots s_n$, and the distribution of this position is given by the $n$-fold convolution of $\mu$. We equip the \emph{path space} $G^{\mathbb{N}}$ with the measure $\mathbb{P}$ defined as the image of the product measure $\mu^{\otimes\mathbb{N}}$ under the map $(s_i)_{i\in\mathbb{N}}\mapsto (g_i)_{i\in\mathbb{N}}$.  

Random walks on mapping class groups have first been studied by Kaimanovich and Masur, whose seminal paper \cite{KM96} has been a main source of inspiration for our work. Given a probability measure $\mu$ on the mapping class group $Mod(S)$ of a surface $S$, whose support generates a nonelementary subgroup of $Mod(S)$, Kaimanovich and Masur have shown that $\mathbb{P}$-almost every sample path of the random walk on $(Mod(S),\mu)$ converges to a uniquely ergodic minimal measured foliation in the Thurston boundary $\mathcal{PMF}$ of $Teich(S)$, and that the hitting measure $\nu$ is the only $\mu$-stationary measure on $\mathcal{PMF}$. Using Reynolds' study of \emph{arational trees} in the boundary of outer space \cite{Rey12} as an analogue for minimal foliations in the boundary of Teichmüller spaces, we partly translate Kaimanovich and Masur's work to the $\text{Out}(F_N)$ case. A tree $T\in\partial CV_N$ is \emph{arational} if every proper free factor of $F_N$ acts freely and discretely on its minimal subtree in $T$. Arational trees are either free (and indecomposable) actions of $F_N$, or they are dual to an arational lamination on a surface having a single boundary component \cite{Rey12}. 

Associated to any $T\in\overline{CV_N}$ with dense orbits is a \emph{simplex} of length measures \cite{Gui00}, which describes the collection of all possible metrics on the topological tree underlying $T$. We denote by $\mathcal{AT}$ the space of equivalence classes of arational trees, and by $\mathcal{FI}$ the space of equivalence classes of free arational trees, under the equivalence relation that identifies two trees that belong to the same simplex. A tree is \emph{uniquely ergometric} if its simplex is reduced to a point. Uniquely ergometric trees provide an analogue of uniquely ergodic foliations on surfaces. We don't know whether sample paths of random walks on $\text{Out}(F_N)$ almost surely converge to uniquely ergometric trees. However, we shall prove the following statement. We define a subgroup $H\subseteq\text{Out}(F_N)$ to be \emph{nonelementary} if the $H$-orbits of all proper free factors of $F_N$, of all projective arational trees, and of all conjugacy classes of elements of $F_N$, are infinite.

\begin{theo} \label{intro-convergence-drift}
Let $\mu$ be a probability measure on $\text{Out}(F_N)$, whose support generates a nonelementary subgroup of $\text{Out}(F_N)$. For $\mathbb{P}$-a.e. sample path $\mathbf{g}:=(g_n)_{n\in\mathbb{N}}$ of the random walk on $(\text{Out}(F_N),\mu)$, there exists a simplex $\xi(\mathbf{g})\in\mathcal{FI}$ such that for all $T_0\in CV_N$, the sequence $(g_n T_0)_{n\in\mathbb{N}}$ converges to $\xi(\mathbf{g})$. The hitting measure is nonatomic, and it is the only $\mu$-stationary measure on $\mathcal{FI}$.
\end{theo}

We then show (under some further assumptions on the measure $\mu$) that $\mathcal{FI}$, equipped with the hitting measure $\nu$, is the Poisson boundary of $(\text{Out}(F_N),\mu)$. Theorem \ref{intro-convergence-drift} ensures that $(\mathcal{FI},\nu)$ is the typical example of a $\mu$-boundary. A \emph{$\mu$-boundary} is a probability space $(B,\nu)$, which is the quotient of the path space $(G^{\mathbb{N}},\mathbb{P})$ with respect to some shift-invariant and $G$-invariant measurable partition (in particular $\nu=bnd_{\ast}\mathbb{P}$, where $bnd:G^{\mathbb{N}}\to B$ is the projection map).  

A $\mu$-boundary $(B,\nu)$ is a \emph{Poisson boundary} if it is maximal, i.e. every $\mu$-boundary is the quotient of $(B,\nu)$ under some $G$-invariant measurable partition. If we equip the path space $G^{\mathbb{N}}$ with the measure $\mathbb{P}_m$ corresponding to an initial distribution for the random walk given by the counting measure on $G$, then the space of ergodic components of the shift in $G^{\mathbb{N}}$ is an abstract realization of the Poisson boundary of $(G,\mu)$. Given a group $G$ equipped with a probability measure $\mu$, a natural question is that of identifying the Poisson boundary of $(G,\mu)$ with some "concrete" $G$-space (which will usually be a topological space, although the Poisson boundary does not carry any intrinsic topology, and is only defined as a measure space). One motivation for this question comes from the problem of understanding bounded $\mu$-harmonic functions on $G$. Indeed, when $(B,\nu)$ is the Poisson boundary of $(G,\mu)$, the formula

\begin{displaymath}
f(g)=\int_{B}\widehat{f}(x)dg_{\ast}\nu(x)
\end{displaymath}

\noindent gives an isometry between the Banach space of $\mu$-essentially bounded $\mu$-harmonic functions on $X$, and $L^{\infty}(B)$. Our main result is the following. 

\begin{theo} \label{intro-Poisson}
Let $\mu$ be a probability measure on $\text{Out}(F_N)$, whose support is finite and generates a nonelementary subgroup of $\text{Out}(F_N)$, and let $\nu$ be the hitting measure on $\mathcal{FI}$. Then the measure space $(\mathcal{FI},\nu)$ is the Poisson boundary of $(\text{Out}(F_N),\mu)$.
\end{theo} 

Theorem \ref{intro-Poisson} is actually true under more general assumptions on the measure $\mu$ (finiteness of the support can be replaced by a finite first logarithmic moment condition with respect to the word metric on $\text{Out}(F_N)$, and a finite entropy condition, see Theorem \ref{Poisson}).

In \cite{Kai00}, Kaimanovich has developed tools coming from entropy theory to prove that a $\mu$-boundary is the Poisson boundary. In particular, he provides a "strip criterion" which requires considering a $\mu$-boundary $B_+$ simultaneously with a $\check{\mu}$-boundary $B_-$ (where $\check{\mu}$ is the probability measure on $G$ defined by $\check{\mu}(g):=\mu(g^{-1})$ for all $g\in G$), and assigning to almost every pair of points in $B_-\times B_+$ a "strip" contained in $G$, which is sufficiently thin in the sense that its intersection with balls for a word metric on $G$ grows subexponentially with the radius of the ball. Given a probability measure $\mu$ on the mapping class group $Mod(S)$ of a surface $S$, satisfying some finiteness conditions, and whose support generates a subgroup of $Mod(S)$ that contains two independent pseudo-Anosov homeomorphisms, Kaimanovich and Masur have shown that $(\mathcal{PMF},\nu)$ is the Poisson boundary of $(Mod(S),\mu)$, by using strips coming from Teichmüller geodesics \cite[Theorem 2.3.1]{KM96}.

Our definition of the strips is based on a simplified version of Hamenstädt's construction of \emph{lines of minima} in outer space \cite{Ham12-2}. We now provide an outline of this construction. There is a natural length pairing between trees in $CV_N$ and elements in $F_N$, defined by letting $\langle T,g\rangle$ be the translation length of $g$ in $T$. Kapovich and Lustig have shown \cite{KL09} that this length pairing extends to an intersection form between trees and \emph{geodesic currents}, which were introduced by Kapovich in \cite{Kap05,Kap06}. Given trees $T\in cv_N$ and $T'\in\overline{cv_N}$ (in $\overline{cv_N}$, trees are considered up to isometry, instead of homothety), and a pair $(\eta,\eta')$ of geodesic currents, we define 

\begin{displaymath}
\Lambda_{\eta,\eta'}(T,T'):=\max\{\frac{\langle T',\eta\rangle}{\langle T,\eta\rangle},\frac{\langle T',\eta'\rangle}{\langle T,\eta'\rangle}\}.
\end{displaymath}

This measures the maximal stretch of $\eta$ and $\eta'$ from $T$ to $T'$. Denoting by $\text{Lip}(T,T')$ the smallest Lipschitz constant of an $F_N$-equivariant map from $T$ to $T'$, we always have

\begin{displaymath}
\Lambda_{\eta,\eta'}(T,T')\le \text{Lip}(T,T'),
\end{displaymath}

\noindent and White has shown that we can always find a "candidate" element $g\in F_N$ whose stretch from $T$ to $T'$ is equal to $\text{Lip}(T,T')$ (and we can even choose $g$ among a finite set of elements of $F_N$ that only depends on the tree $T$), see \cite{FM11}. 

For "generic" pairs $(\eta,\eta')$ of currents, we have $\Lambda_{\eta,\eta'}(T,T')>0$, and for all $L\ge 1$, we define the \emph{$L$-axis} of the pair $(\eta,\eta')$ as the set of all trees in $CV_N$ for which 

\begin{displaymath}
1\le\frac{\text{Lip}(T,T')}{\Lambda_{\eta,\eta'}(T,T')}\le L 
\end{displaymath}

\noindent for all $T'\in\overline{CV_N}$. In other words, a tree $T\in CV_N$ is in the $L$-axis of $(\eta,\eta')$ if the stretch of either $\eta$ or $\eta'$ gives a good estimate of the Lipschitz distortion from $T$ to any $T'\in\overline{CV_N}$, up to an error controlled by $L$ (or informally, if the pair $(\eta,\eta')$ is a fairly good pair of "candidates" for the tree $T$). Following Hamenstädt's arguments \cite{Ham12-2}, we show that these axes are close to being geodesics in $CV_N$ for the symmetric Lipschitz metric (see \cite{FM11} for an introduction to this metric), although they may contain holes (notice that the $L$-axis of a pair $(\eta,\eta')$ can even be empty if $L$ is too small). This will be the key point for checking the growth condition on the strips.

Associated to any arational tree $T$ is a finite collection of "ergodic" currents $\text{Erg}(T)$. This enables us to associate an $L$-strip in $\text{Out}(F_N)$ to almost every pair of trees $(T_-,T_+)\in\mathcal{FI}\times\mathcal{FI}$. We then show that we can choose $L$ in a uniform way to ensure that the strips are almost surely nonempty. 
\\
\\
\indent Using recent work of Bestvina and Reynolds \cite{BR13} and Hamenstädt \cite{Ham12}, our results can be interpreted in terms of the free factor complex and its Gromov boundary. When $N\ge 3$, the \emph{free factor complex} $\mathcal{FF}_N$ is the simplicial complex whose vertices are the conjugacy classes of proper free factors of $F_N$, and higher dimensional simplices correspond to chains of inclusion of free factors (one has to slightly modify the definition when $N=2$ to ensure that $\mathcal{FF}_2$ is connected). It was proven to be Gromov hyperbolic by Bestvina and Feighn \cite{BF12}, see also \cite{KR12} for an alternative proof. Its Gromov boundary was identified by Bestvina and Reynolds \cite{BR13} and Hamenstädt \cite{Ham12} with the space of simplices of arational trees in $\partial CV_N$. Using their work, Theorems \ref{intro-convergence-drift} and \ref{intro-Poisson} lead to the following statement.

\begin{theo}\label{intro-factor}
Let $\mu$ be a probability measure on $\text{Out}(F_N)$, whose support generates a nonelementary subgroup of $\text{Out}(F_N)$. Then for $\mathbb{P}$-almost every sample path $\mathbf{g}:=(g_n)_{n\in\mathbb{N}}$ of the random walk on $(\text{Out}(F_N),\mu)$, there exists $\xi(\mathbf{g})\in\partial\mathcal{FF}_N$, such that for all $x\in\mathcal{FF}_N$, the sequence $(g_n x)_{n\in\mathbb{N}}$ converges to $\xi(\mathbf{g})$. The hitting measure $\nu$ on $\partial\mathcal{FF}_N$ is the unique $\mu$-stationary measure on $\partial\mathcal{FF}_N$. If in addition, the measure $\mu$ has finite support, then $(\partial\mathcal{FF}_N,\nu)$ is the Poisson boundary of $(\text{Out}(F_N),\mu)$. 
\end{theo}

The convergence statement was obtained with different methods by Calegari and Maher, in the more general context of groups acting on (possibly nonproper) Gromov hyperbolic spaces \cite[Theorem 5.34]{CM13}.

\section*{Acknowledgements}

I warmly thank my advisor Vincent Guirardel, whose advice led to significant improvements in the exposition of the proof.

\section{Preliminaries on $\text{Out}(F_N)$ and related complexes}

\subsection{Outer space}
Let $N\ge 2$. \emph{Outer space} $CV_N$ is defined to be the space of simplicial free, minimal, isometric actions of $F_N$ on simplicial metric trees, up to $F_N$-equivariant homotheties \cite{CV86} (an $F_N$-action on a tree is \emph{minimal} if there is no proper invariant subtree). We denote by $cv_N$ the \emph{unprojectivized outer space}, in which trees are considered up to $F_N$-equivariant isometries, instead of homotheties. The group $\text{Out}(F_N)$ acts on $CV_N$ and on $cv_N$ on the left by setting $\Phi(T,\rho)=(T,\rho\circ \phi^{-1})$ for all $\Phi\in\text{Out}(F_N)$, where $\rho:F_N\to\text{Isom}(T)$ denotes the action, and $\phi\in\text{Aut}(F_N)$ is any lift of $\Phi$ to $\text{Aut}(F_N)$. 

An \emph{$\mathbb{R}$-tree} is a metric space $(T,d_T)$ in which any two points $x$ and $y$ are joined by a unique arc, which is isometric to a segment of length $d_T(x,y)$. Let $T$ be an \emph{$F_N$-tree}, i.e. an $\mathbb{R}$-tree equipped with an isometric action of $F_N$. For $g\in F_N$, the \emph{translation length} of $g$ in $T$ is defined to be

\begin{displaymath}
||g||_T:=\inf_{x\in T}d_T(x,gx).
\end{displaymath}

\noindent \noindent Culler and Morgan have shown in \cite[Theorem 3.7]{CM87} that the map

\begin{displaymath}
\begin{array}{cccc}
i:&cv_N&\to &\mathbb{R}^{F_N}\\
&T&\mapsto & (||g||_T)_{g\in F_N}
\end{array}
\end{displaymath}

\noindent is an embedding, whose image has projectively compact closure $\overline{CV_N}$ \cite[Theorem 4.5]{CM87}. Bestvina and Feighn \cite{BF94}, extending results by Cohen and Lustig \cite{CL95}, have characterized the points of this compactification as being the minimal, \emph{very small} $F_N$-trees, i.e. the $F_N$-trees with trivial or maximally cyclic arc stabilizers and trivial tripod stabilizers. We denote by $\overline{cv_N}$ the lift of $\overline{CV_N}$ to $\mathbb{R}^{F_N}$.  

\subsection{Algebraic laminations and currents} \label{sec-laminations}

Let $\partial^2F_N:=\partial F_N\times\partial F_N\smallsetminus\Delta$, where $\Delta$ is the diagonal. Denote by $i:\partial^2 F_N\to\partial^2F_N$ the involution that exchanges the factors. An \emph{algebraic lamination} is a nonempty, closed, $F_N$-invariant and $i$-invariant subset of $\partial^2F_N$. Any nontrivial element $g\in F_N$ determines an element $(g^{-\infty},g^{+\infty})\in\partial^2F_N$ by setting $g^{-\infty}:=\lim_{n\to +\infty}g^{-n}$ and $g^{+\infty}:=\lim_{n\to +\infty}g^n$. Let $T\in\overline{cv_N}$. For $\epsilon>0$, let 

\begin{displaymath}
L_{\epsilon}(T):=\overline{\{(g^{-\infty},g^{+\infty})|||g||_T<\epsilon\}}.
\end{displaymath}

\noindent Then 

\begin{displaymath}
L(T):=\bigcap_{\epsilon>0}L_{\epsilon}(T)
\end{displaymath}

\noindent is an algebraic lamination, called the lamination \emph{dual} to the tree $T$ (see \cite{CHL08-1,CHL08-2} for an extended study of algebraic laminations). Notice that $L(T)$ only depends on the projective class of the tree $T$, and hence can be defined for $T\in\overline{CV_N}$.  

A \emph{current} on $F_N$ is an $F_N$-invariant Borel measure on $\partial^2F_N$ that is finite on compact subsets of $\partial^2F_N$. The systematic study of currents on $F_N$ was initiated by Kapovich \cite{Kap05,Kap06}. We denote by $Curr_N$ the set of currents on $F_N$, equipped with the weak-$\ast$ topology, and by $\mathbb{P}Curr_N$ the space of projective classes (i.e. homothety classes) of currents. The space $\mathbb{P}Curr_N$ is compact \cite[Proposition 2.5]{Kap06}. 

To every $g\in F_N$ which is not of the form $h^k$ for any $h\in F_N$ and $k>1$ (we say that $g$ is not a \emph{proper power}), one associates a \emph{rational current} $\eta_g$ by letting $\eta_g(S)$ be the number of translates of $(g^{-\infty},g^{+\infty})$ that belong to $S$ for all closed-open subsets $S\subseteq\partial^2F_N$, see \cite[Definition 5.1]{Kap06} (for the case of proper powers one may set $\eta_{h^k}:=k\eta_h$). The group $\text{Out}(F_N)$ acts on $Curr_N$ on the left in the following way \cite[Proposition 2.15]{Kap06}. Given a compact set $K\subseteq\partial^2F_N$, an element $\Phi\in\text{Out}(F_N)$, and a current $\eta\in Curr_N$, we set $\Phi(\eta)(K):=\eta(\phi^{-1}(K))$, where $\phi\in\text{Aut}(F_N)$ is any representative of $\Phi$. In \cite[Section 5]{Kap06}, Kapovich defined an intersection form between elements of $cv_N$ and currents, which was then extended by Kapovich and Lustig to trees in $\overline{cv_N}$ \cite{KL09}.

\begin{theo} \label{intersection-form} (Kapovich-Lustig \cite[Theorem A]{KL09})
There exists a unique $\text{Out}(F_N)$-invariant continuous function 

\begin{displaymath}
\langle .,.\rangle : \overline{cv_N}\times Curr_N\to\mathbb{R}_+
\end{displaymath}

\noindent which is $\mathbb{R}_+$-homogeneous in the first coordinate and $\mathbb{R}_+$-linear in the second coordinate, and such that for all $T\in\overline{cv_N}$, and all $g\in F_N\smallsetminus\{e\}$, we have $\langle T,\eta_g\rangle = ||g||_T$.
\end{theo}

Kapovich and Lustig give the following characterization of zero intersection. 

\begin{theo} \label{currents-duality} (Kapovich-Lustig \cite[Theorem 1.1]{KL10})
For all $T\in\overline{cv_N}$ and all $\eta\in Curr_N$, we have $\langle T,\eta\rangle=0$ if and only if $\text{Supp}(\eta)\subseteq L(T)$. In particular, for all $T\in cv_N$ and all $\eta\in Curr_N$, we have $\langle T,\eta\rangle\neq 0$, while for all $T\in\partial cv_N$, there exists $\eta\in Curr_N$ such that $\langle T,\eta\rangle=0$.
\end{theo}

A projective current $[\eta]\in \mathbb{P}Curr_N$ is \emph{ergodic} if for every $F_N$-invariant measurable subset $S\subseteq\partial^2 F_N$, we either have $\eta(S)=0$ or $\eta(\partial^2 F_N\smallsetminus S)=0$. We denote by $\text{Erg}_N$ the space of ergodic currents, which coincides with the set of extreme points of the compact convex space $\mathbb{P}Curr_N$. Given an $F_N$-tree $T$, we denote by $\text{Dual}(T)$ the space of all projective currents $[\eta]\in\mathbb{P}Curr_N$ such that $\langle T,\eta\rangle=0$ (this makes sense since nullity of $\langle T,\eta\rangle$ only depends on the projective class of $\eta$). For all $T\in\overline{cv_N}$, the space $\text{Dual}(T)$ is a compact convex subspace of $\mathbb{P}Curr_N$. Equivariance of the intersection form implies that for all $T\in\overline{cv_N}$ and $\Phi\in\text{Out}(F_N)$, we have $\Phi \text{Dual}(T)=\text{Dual}(\Phi T)$. The extreme points of $\text{Dual}(T)$ are the ergodic currents which are dual to $T$. Denoting by $\text{Erg}(T)$ the set of such ergodic currents, for all $\Phi\in\text{Out}(F_N)$, we have $\text{Erg}(\Phi T)=\Phi \text{Erg}(T)$. Coulbois and Hilion have shown that $\text{Dual}(T)$ is finite-dimensional as soon as the $F_N$-action on $T$ is free and has dense orbits \cite{CH13}.

\begin{theo} \label{dimension-currents} (Coulbois-Hilion \cite[Theorem 1.1]{CH13})
Let $T$ be an $\mathbb{R}$-tree with a free, minimal action of $F_N$ by isometries with dense orbits. Then $\text{Dual}(T)$ contains at most $3N-5$ projective classes of ergodic currents.
\end{theo}

\subsection{The Lipschitz metric on outer space}\label{sec-metric}

Outer space is equipped with an asymmetric metric: the distance $d(T,T')$ between two trees $T,T'\in CV_N$ is defined as the logarithm of the infimal Lipschitz constant $\text{Lip}(T,T')$ of an $F_N$-equivariant map from the covolume $1$ representative of $T$ to the covolume $1$ representative of $T'$, see \cite{FM11}. One can symmetrize the metric on $CV_N$ by setting $d_{sym}(T,T'):=d(T,T')+d(T',T)$. The Lipschitz metric on $CV_N$ can be interpreted in terms of the intersection pairing between trees and currents. Given a subset $X\subseteq \mathbb{P}Curr_N$, we let

\begin{displaymath}
\Lambda_{X}(T,T'):=\sup_{[\eta]\in X}\frac{\langle T',\eta\rangle}{\langle T,\eta\rangle}.
\end{displaymath}

\begin{theo} (White, see \cite[Proposition 2.3]{AK11}, \cite[Proposition 3.15]{FM11} or \cite[Lemma 4.1]{Ham12-2})\label{White}
For all $T,T'\in CV_N$, we have $\text{Lip}(T,T')=\Lambda_{\mathbb{P}Curr_N}(T,T')$.
\end{theo}

Algom-Kfir has shown in \cite[Proposition 4.5]{AK12} that the equality stated in Theorem \ref{White} also holds when $T'\in\overline{CV_N}$ (she actually states her result when $T'$ belongs to the metric completion of $CV_N$). Notice that equality between $\text{Lip}(T,T')$ and $\Lambda_{\mathbb{P}Curr_N}(T,T')$ does not depend on the choice of a representative of the projective classes of $T$ and $T'$.

\subsection{Arational trees} \label{sec-arational}

Let $H\le F_N$ be a finitely generated subgroup of $F_N$. The boundary $\partial H$ naturally embeds in $\partial F_N$. We say that $H$ \emph{carries a leaf} of an algebraic lamination $L$ if $L\cap\partial^2 H\neq\emptyset$. A tree $T\in\partial CV_N$ is \emph{arational} if no leaf of $L(T)$ is carried by a proper free factor of $F_N$. We denote by $\widetilde{\mathcal{AT}}$ the subspace of $\partial CV_N$ consisting of arational trees. Arational trees have dense orbits, and Reynolds has shown that arational trees are either free (and indecomposable) or dual to an arational measured lamination on a surface with one boundary component \cite[Theorem 1.1]{Rey12}. We denote by $\widetilde{\mathcal{FI}}$ (standing for "free indecomposable") the subspace of $\widetilde{\mathcal{AT}}$ consisting of free actions of $F_N$. Let $\sim$ be the equivalence relation on $\widetilde{\mathcal{AT}}$ defined by $T\sim T'$ if $L(T)=L(T')$. Two trees $T,T'\in\widetilde{\mathcal{AT}}$ are equivalent if and only if they belong to the same \emph{simplex} of length measures in $\partial CV_N$ (see \cite[Section 5]{Gui98} for definitions), i.e. they have the same underlying topological tree, see \cite{CHL07}. Let $\mathcal{AT}:=\widetilde{\mathcal{AT}}/{\sim}$ and $\mathcal{FI}:=\widetilde{\mathcal{FI}}/{\sim}$. Classes of the relation $\sim$ are compact subspaces of $\partial CV_N$ \cite[Lemma 7.1]{BR13}. By definition of the relation $\sim$, and thanks to Theorem \ref{currents-duality}, it makes sense to define $\text{Dual}(T)$ and $\text{Erg}(T)$ for $T\in\mathcal{FI}$. Theorem \ref{dimension-currents} therefore implies the following fact. 

\begin{prop}\label{dual-arational}
For all $T\in\mathcal{FI}$, the set $\text{Erg}(T)$ is finite, of cardinality at most $3N-5$.
\qed
\end{prop}

The following unique duality statement is a version of a theorem due Bestvina and Reynolds \cite[Theorem 4.4]{BR13} and Hamenstädt \cite[Corollary 10.6]{Ham12}.  

\begin{theo}\label{arational-duality}
Let $T_1\in\widetilde{\mathcal{FI}}$, and let $\eta\in Curr_N$ be such that $\langle T_1,\eta\rangle=0$. If $T_2\in\partial CV_N$ also satisfies $\langle T_2,\eta\rangle=0$, then $T_2\in\widetilde{\mathcal{FI}}$ and $T_1\sim T_2$.
\end{theo}

\begin{proof}
By Theorem \ref{currents-duality}, as $\langle T_1,\eta\rangle=0$, we have $\text{Supp}(\eta)\subseteq L(T_1)$. If $\text{Supp}(\eta)$ contained a periodic leaf (whose $F_N$-translates form the support of a rational current $\eta_g$ for some $g\in F_N$), then we would have $||g||_{T_1}=0$, contradicting freeness of the $F_N$-action on $T_1$. In addition, the support of a current cannot contain isolated nonperiodic leaves, since translates of such leaves have accumulation points, and currents are Radon measures. This implies that $\text{Supp}(\eta)$ does not contain any isolated leaf. Therefore $\text{Supp}(\eta)$ is contained in the derived lamination $L'(T_1)$ (i.e. the sublamination of $L(T_1)$ consisting of non-isolated leaves). Since $T_1\in\widetilde{\mathcal{FI}}$, by \cite[Proposition 4.2]{BR13}, the lamination $L'(T_1)$ is minimal (i.e. it does not contain any proper sublamination), so $\text{Supp}(\eta)=L'(T_1)$. Since we also have $\langle T_2,\eta\rangle=0$, Theorem \ref{currents-duality} implies that $L'(T_1)\subseteq L(T_2)$. 

If $T_2$ does not have dense orbits, then all leaves of $L(T_2)$ are carried by a vertex group of the canonical decomposition of $T_2$ as a graph of actions with dense orbits (see \cite{KL09}). Such vertex groups have infinite index in $F_N$. However, as $T_1$ is free and indecomposable, a theorem of Reynolds \cite{Rey11} shows that no leaf of $L(T_1)$ is carried by an infinite index subgroup of $F_N$. This yields a contradiction.

Therefore, the tree $T_2$ has dense orbits, and it follows from \cite[Section 8]{CHL08-2} that $L(T_2)$ is diagonally closed. By \cite[Proposition 4.2]{BR13}, the lamination $L(T_1)$ is the diagonal closure of $L'(T_1)$. Hence we have $L(T_1)\subseteq L(T_2)$. Since $T_1$ is indecomposable, this implies that $L(T_1)=L(T_2)$ by \cite[Proposition 3.1]{BR13}, and $T_2\in\widetilde{\mathcal{FI}}$.   
\end{proof}

Following Hamenstädt \cite[Section 3]{Ham12-2}, we say that a pair $(\eta,\eta')\in Curr_N^2$ is \emph{positive} if for all $T\in\overline{cv_N}$, we have $\langle T,\eta+\eta'\rangle>0$ (this again makes sense when $[\eta],[\eta']\in\mathbb{P}Curr_N$). Denote by $\Delta$ the diagonal in $\mathcal{FI}\times\mathcal{FI}$. As a consequence of Proposition \ref{dual-arational} and Theorem \ref{arational-duality}, we get the following fact.

\begin{cor}\label{arational-current}
For all pairs $(T,T')\in\mathcal{FI}\times\mathcal{FI}\smallsetminus\Delta$, and all $(\eta,\eta')\in\text{Dual}(T)\times\text{Dual}(T')$, the pair $(\eta,\eta')$ is positive. In particular, the set $\text{Erg}(T)\times\text{Erg}(T')$ is a finite set of positive pairs of projective currents.
\qed
\end{cor}

\section{Random walks in $\text{Out}(F_N)$} \label{sec-random}

In this section, all topological spaces are equipped with their Borel $\sigma$-algebra. Let $G$ be a countable group, and $\mu$ a probability measure on $G$. We denote by $gr(\mu)$ the subgroup of $G$ generated by the support of the measure $\mu$. The \emph{random walk} on $G$ with respect to the measure $\mu$ is the Markov chain on $G$ with transition probabilities $p(x,y):=\mu(x^{-1}y)$. The \emph{step space} for the random walk is the product probability space $(G^{\mathbb{N}},\mu^{\otimes\mathbb{N}})$. The position of the random walk at time $n$ is given from its position $g_0=e$ at time $0$ by multiplying on the right by a sequence of independent $\mu$-distributed increments $s_i$, i.e. $g_n=s_1\dots s_n$. So the distribution of the location of the random walk at time $n$ is given by the $n$-fold convolution of the measure $\mu$, which we denote by $\mu^{\ast n}$. We equip the \emph{path space} $G^{\mathbb{N}}$ with the $\sigma$-algebra $\mathcal{A}$ generated by the cylinders $\{\textbf{g}\in G^{\mathbb{N}}|g_i=g\}$ for all $i\in\mathbb{N}$ and all $g\in G$. We denote by $\mathbb{P}$ the probability measure on the path space $G^{\mathbb{N}}$ induced by the map $(s_1,s_2,\dots)\mapsto (g_1,g_2,\dots)$. 

Let $\mu$ be a probability measure on $\text{Out}(F_N)$. We aim at understanding the asymptotic behaviour of the random walk on $\text{Out}(F_N)$ with respect to the measure $\mu$. A subgroup $H\subseteq\text{Out}(F_N)$ is \emph{nonelementary} if the $H$-orbits of all proper free factors of $F_N$, of all projective arational trees, and of all conjugacy classes of elements of $F_N$, are infinite. Arguing as in \cite{Hor14-3} (see also \cite{HM09}), one can show that this is equivalent to $H$ containing two independent atoroidal fully irreducible elements (we will not use this fact in the sequel). The following result is a consequence of Propositions \ref{existence-stationary-ff}, \ref{convergence-cvn} and \ref{unique}.

\begin{theo} \label{convergence-drift}
Let $\mu$ be a probability measure on $\text{Out}(F_N)$, whose support generates a nonelementary subgroup of $\text{Out}(F_N)$. For $\mathbb{P}$-a.e. sample path $\mathbf{g}:=(g_n)_{n\in\mathbb{N}}$ of the random walk on $(\text{Out}(F_N),\mu)$, there exists a simplex $\xi(\mathbf{g})\in\mathcal{FI}$ such that for all $T_0\in CV_N$, the sequence $(g_n T_0)_{n\in\mathbb{N}}$ converges to $\xi(\mathbf{g})$. The hitting measure is nonatomic, and it is the only $\mu$-stationary measure on $\mathcal{FI}$.
\end{theo}

\begin{question}
Is it true that $\mathbb{P}$-a.e. sample path of the random walk on $(\text{Out}(F_N),\mu)$ converges to a \emph{uniquely ergometric} tree, i.e. a tree whose corresponding simplex consists of a single element, as in the case of mapping class groups \cite[Theorem 2.2.4]{KM96} ? One could also ask the "dual" question of \emph{unique ergodicity}, in the sense that there exists a unique current dual to the tree $T$, for limit points of sample paths of the random walk. It is known that the attracting tree in $\partial CV_N$ of a nongeometric fully irreducible element of $\text{Out}(F_N)$ is uniquely ergodic \cite[Proposition 5.6]{CHL08-3}. As generic elements of $\text{Out}(F_N)$ are fully irreducible and nongeometric, it seems reasonable to hope for such a result. However, Kaimanovich and Masur's argument in the case of mapping class groups relies on a theorem of Masur stating that any Teichmüller geodesic whose vertical foliation is minimal but not uniquely ergodic has to leave the thick part of the Teichmüller space of the associated surface \cite{Mas92}, and we do not know any good analogue of this theorem for outer space.
\end{question}

\begin{rk}
If we remove the condition on orbits of conjugacy classes of elements of $F_N$ in the definition of nonelementary subgroups, we still get convergence of almost every sample path to an element of $\mathcal{AT}$. However, if $gr(\mu)$ is nonelementary in this new sense, and virtually fixes the conjugacy class of an element in $F_N$, then it is virtually a subgroup of the mapping class group of a compact surface with a single boundary component. This case is already covered by Kaimanovich and Masur's work \cite{KM96}.
\end{rk}

\subsection{Stationary measures on $\partial CV_N$}

The following proposition was essentially proved in \cite[Proposition 3.2]{Hor14-3}, without the assumption that $gr(\mu)$ does not preserve any finite set of conjugacy classes of elements of $F_N$. By the same reasoning as in the proof in \cite{Hor14-3}, we will show this extra assumption implies that every $\mu$-stationary measure is concentrated on the set of free actions. Measurability of $\widetilde{\mathcal{AT}}$ was proved in \cite[Lemma 3.4]{Hor14-3}, and measurability of $\widetilde{\mathcal{FI}}$ follows since freeness of the action is a measurable condition.

\begin{prop} \label{stationary-on-outer-space}
Let $\mu$ be a probability measure on $\text{Out}(F_N)$, such that $\text{gr}(\mu)$ is nonelementary. Then every $\mu$-stationary Borel probability measure on $\partial CV_N$ is purely nonatomic and concentrated on $\widetilde{\mathcal{FI}}$. 
\end{prop} 

The proof of Proposition \ref{stationary-on-outer-space} is based on the following classical statement, whose proof relies on a maximum principle argument.

\begin{lemma} \label{disjoint-translations} (Ballmann \cite{Bal89}, Kaimanovich-Masur \cite[Lemma 2.2.2]{KM96}, Woess \cite[Lemma 3.4]{Woe89})
Let $\mu$ be a probability measure on a countable group $G$, and let $\nu$ be a $\mu$-stationary measure on a $G$-space $X$. Suppose $E\subset X$ is a measurable subset such that for all $g\in gr(\mu)$, either $gE=E$ or $gE\cap E=\emptyset$, and the $gr(\mu)$-orbit of $E$ is infinite. Then $\nu(E)=0$.
\end{lemma}

\begin{proof}[Proof of Proposition \ref{stationary-on-outer-space}]
Let $\nu$ be a $\mu$-stationary measure on $\partial CV_N$. The fact that $\nu(\widetilde{\mathcal{AT}})=1$ was proved in \cite[Lemma 3.4]{Hor14-3}. Nonatomicity of $\nu$ follows from Lemma \ref{disjoint-translations} applied to the singleton $E=\{T\}$, where $T\in\widetilde{\mathcal{AT}}$, since nonelementarity of $gr(\mu)$ implies that the $gr(\mu)$-orbit of $T$ is infinite. 

Let $X$ be a finite set of conjugacy classes of elements of $F_N$. The set $E_X$ of trees in $\partial CV_N$ whose collection of cyclic point stabilizers is equal to $X$ is measurable, see \cite[Lemma 3.5]{Hor14-3}, and nonelementarity of $gr(\mu)$ implies that the $gr(\mu)$-orbit of $E_X$ is infinite. As arational trees which are not free are dual to an arational measured foliation on a surface with one boundary component \cite[Theorem 1.1]{Rey12}, for which the boundary curve is the only point stabilizer, Proposition \ref{stationary-on-outer-space} follows from Lemma \ref{disjoint-translations} applied to the sets $E_X$.
\end{proof}

As a consequence of Proposition \ref{stationary-on-outer-space}, we get the following result.  

\begin{prop}\label{existence-stationary-ff}
Let $\mu$ be a probability measure on $\text{Out}(F_N)$, whose support generates a nonelementary subgroup of $\text{Out}(F_N)$. Then there exists a $\mu$-stationary probability measure on $\mathcal{FI}$, and all such measures are nonatomic.
\end{prop}

\begin{proof}
The first part of the statement is a consequence of Proposition \ref{stationary-on-outer-space}. Nonatomicity is proved as above, by noticing that if $gr(\mu)$ virtually fixes a simplex in $\mathcal{FI}$, then it also preserves the set of extremal points of this simplex in $\overline{CV_N}$, which is finite by \cite[Corollary 5.4]{Gui00}.
\end{proof}

As any $\mu$-stationary measure on $\partial CV_N$ projects to a nonatomic $\mu$-stationary measure on $\mathcal{FI}$, we get the following result. Notice that $\sim$-classes are compact subsets of $\partial CV_N$, in particular they are measurable.

\begin{prop} \label{nonatomic}
Let $\mu$ be a probability measure on $\text{Out}(F_N)$, whose support generates a nonelementary subgroup of $\text{Out}(F_N)$, and let $\nu$ be a $\mu$-stationary probability measure on $\partial CV_N$. Then every class of the relation $\sim$ on $\widetilde{\mathcal{AT}}$ has $\nu$-measure $0$.
\qed
\end{prop}

\subsection{Limit points of random walks on $CV_N$, and convergence to $\mathcal{FI}$}

Adapting Kaimanovich and Masur's argument from \cite[Section 1.5]{KM96} to the $\text{Out}(F_N)$ case, we now study the possible limit points of sequences $(g_nT)_{n\in\mathbb{N}}$, where $(g_n)_{n\in\mathbb{N}}$ is a sequence of elements of $\text{Out}(F_N)$ which tends to infinity, and $T\in\overline{CV_N}$ is an $F_N$-tree. We recall that whenever $X$ is a Borel space, a sequence of measures $(\nu_n)_{n\in\mathbb{N}}$ on $X$ \emph{weakly converges} to a measure $\nu$ if $(\nu_n(f))_{n\in\mathbb{N}}$ converges to $\nu(f)$ for every bounded continuous real-valued function on $X$. The goal of the present section is to prove the following result. The convergence statement in Theorem \ref{convergence-drift} is a consequence of Proposition \ref{convergence-cvn}.

\begin{prop} \label{convergence-cvn}
Let $\mu$ be a probability measure on $\text{Out}(F_N)$, whose support generates a nonelementary subgroup of $\text{Out}(F_N)$, and let $\nu$ be a $\mu$-stationary measure on $\partial CV_N$. Then for $\mathbb{P}$-a.e. sample path $\mathbf{g}:=(g_n)_{n\in\mathbb{N}}$ of the random walk on $(\text{Out}(F_N),\mu)$, there is a simplex $\Delta(\mathbf{g})\subseteq\widetilde{\mathcal{FI}}$, such that 

\begin{itemize}
\item the translates $g_n\nu$ weakly  converge to a measure $\lambda(\mathbf{g})$ supported on $\Delta(\mathbf{g})$, and
\item for all $T\in CV_N$, all limit points of the sequence $(g_n T)_{n\in\mathbb{N}}$ belong to $\Delta(\mathbf{g})$.
\end{itemize}
\end{prop}

Our proof of Proposition \ref{convergence-cvn} relies on the following general statement about random walks on countable groups.

\begin{lemma} \label{decomposition-measure} (Furstenberg \cite{Fur73}, Kaimanovich-Masur \cite[Lemma 2.2.3]{KM96})
Let $\mu$ be a probability measure on a countable group $G$, and let $\nu$ be a $\mu$-stationary measure on a compact separable $G$-space. Then for $\mathbb{P}$-a.e. sample path $\mathbf{g}=(g_n)_{n\in\mathbb{N}}$ of the random walk on $(G,\mu)$, the translates $g_n\nu$ converge weakly to a limit $\lambda(\mathbf{g})$, and
\begin{displaymath}
\nu=\int_{G^{\mathbb{N}}} \lambda(\mathbf{g})d\mathbb{P}(\mathbf{g}).
\end{displaymath}
\end{lemma}

We will show the following statement.

\begin{prop} \label{limit-points-2}
Let $\mu$ be a probability measure on $\text{Out}(F_N)$, whose support generates a nonelementary subgroup of $\text{Out}(F_N)$. Let $\nu$ be a $\mu$-stationary probability measure on $\partial CV_N$, and let $(g_n)_{n\in\mathbb{N}}$ be an unbounded sequence of elements of $\text{Out}(F_N)$ such that $g_n\nu$ converges weakly to a measure $\lambda$ on $\partial CV_N$. Then either $\lambda$ is concentrated on $\partial CV_N\smallsetminus \widetilde{\mathcal{FI}}$, or it is concentrated on $\widetilde{\mathcal{FI}}$, on a single class of the relation $\sim$. In the first case, all limit points of sequences $(g_nT)_{n\in\mathbb{N}}$ for $T\in CV_N$ are contained in $\partial CV_N\smallsetminus \widetilde{\mathcal{FI}}$, and in the second case they are all contained in the same class of the relation $\sim$, on which $\lambda$ is concentrated.
\end{prop}

We first explain how to deduce Proposition \ref{convergence-cvn} from Lemma \ref{decomposition-measure} and Proposition \ref{limit-points-2}.

\begin{proof}[Proof of Proposition \ref{convergence-cvn}]
Let $\nu$ be a $\mu$-stationary measure on $\partial CV_N$. As $gr(\mu)$ is nonelementary, the measure $\nu$ is concentrated on $\widetilde{\mathcal{FI}}$ (Proposition \ref{stationary-on-outer-space}). Lemma \ref{decomposition-measure} thus implies that for $\mathbb{P}$-a.e. sample path $\textbf{g}$ of the random walk on $(\text{Out}(F_N),\mu)$, the limit measure $\lambda(\textbf{g})$ exists and is concentrated on $\widetilde{\mathcal{FI}}$. As $\text{gr}(\mu)$ is nonelementary, it is unbounded, so $\mathbb{P}$-a.e. sample path of the random walk is unbounded. Proposition \ref{limit-points-2} implies that for $\mathbb{P}$-a.e. sample path $\textbf{g}=(g_n)_{n\in\mathbb{N}}$ of the random walk, the measure $\lambda(\textbf{g})$ is concentrated on a single $\sim$-class $\Delta(\mathbf{g})$, and for all $T\in CV_N$, all limit points of the sequence $(g_n T)_{n\in\mathbb{N}}$ belong to $\Delta(\mathbf{g})$. 
\end{proof}

We are left showing Proposition \ref{limit-points-2}. We will appeal to another general statement due to Kaimanovich and Masur. We provide a proof for completeness.

\begin{lemma} \label{KM-limit-points} (Kaimanovich-Masur \cite[Lemma 1.5.5]{KM96})
Let $\nu$ be a Borel probability measure on $\partial CV_N$. Let $(g_n)_{n\in\mathbb{N}}\in \text{Out}(F_N)^{\mathbb{N}}$ be a sequence of elements in $\text{Out}(F_N)$ such that $g_n\nu$ converges weakly to a probability measure $\lambda$ on $\partial CV_N$. If there is a Borel subset $E\subseteq \partial CV_N$ with $\nu(E)=1$ and a closed subset $\Omega\subseteq \partial CV_N$ that contains all limit points of sequences $(g_nT)_{n\in\mathbb{N}}$ for $T\in E$, then the measure $\lambda$ is supported on $\Omega$.
\end{lemma}

\begin{proof}
Let $U\subseteq \partial CV_N$ be an open subset that contains $\Omega$. Compactness of $\partial CV_N$ implies the existence for all $T\in E$ of an integer $n(T)$ such that for all $n\ge n(T)$, we have $g_nT\in U$. Let $\epsilon>0$. As $\nu(E)=1$, there exists an integer $N\in\mathbb{N}$ such that $\nu(\{T|n(T)\le N\})\ge 1-\epsilon$. This implies that for all $n\ge N$, we have $g_n\nu(U)\ge 1-\epsilon$, and therefore $g_n\nu(\overline{U})\ge 1-\epsilon$. As $\overline{U}$ is a closed set, weak convergence of the measures $g_n\nu$ implies that $\lambda(\overline{U})\ge 1-\epsilon$, see \cite[Theorem 2.1]{Bil68}. Therefore, we get that for all open neighborhoods $U$ of $\Omega$, we have $\lambda(\overline{U})=1$. By letting $U_n$ be the $\frac{1}{n}$-regular neighborhood of $\Omega$ for any metric on $\partial CV_N$, as $\Omega$ is a closed set, we have

\begin{displaymath}
\Omega=\bigcap_{n\in\mathbb{N}} \overline{U_n},
\end{displaymath}

\noindent which implies that $\lambda(\Omega)=1$.
\end{proof}

To prove Proposition \ref{limit-points-2}, we need to understand possible limit points of sequences in $\overline{CV_N}$. Let $\ast_{CV_N}\in CV_N$. Let $\eta_0\in \mathbb{P}Curr_N$ be such that for all $T\in\overline{CV_N}$, we have $\langle T,\eta_0\rangle >0$ (take for example a basis $\{x_1,\dots,x_N\}$ of $F_N$, and let $\eta_0:=[x_1]+\dots+[x_N]+[x_1x_2]+\dots+[x_1x_N]$). Let $(h_n)_{n\in\mathbb{N}}\in\text{Out}(F_N)^{\mathbb{N}}$, and let $\eta\in\mathbb{P}Curr_N$. The pair $((h_n)_{n\in\mathbb{N}},\eta)$ is \emph{universally converging} if

\begin{itemize}
\item the sequence $(h_n\ast_{CV_N})_{n\in\mathbb{N}}$ converges (projectively) to a tree $T_{\infty}\in\overline{CV_N}$ such that $\langle T_{\infty},\eta\rangle=0$, and
\item the sequence $(h_n^{-1}\eta_0)_{n\in\mathbb{N}}$ converges (projectively) to a current $\eta_0^{\infty}\in\mathbb{P}Curr_N$, and
\item the sequence $(h_n^{-1}\eta)_{n\in\mathbb{N}}$ converges (projectively) to a current $\eta^{\infty}\in\mathbb{P}Curr_N$.
\end{itemize}

\noindent The following lemma follows from compactness of $\mathbb{P}Curr_N$.

\begin{lemma} \label{universal-subsequence}
Let $(g_n)_{n\in\mathbb{N}}$ be an unbounded sequence of elements of $\text{Out}(F_N)$, let $T_{\infty}\in\partial CV_N$ be a limit point of $(g_n\ast_{CV_N})_{n\in\mathbb{N}}$, and $\eta\in Dual(T_{\infty})$. Then there exists a subsequence $(h_n)_{n\in\mathbb{N}}$ of $(g_n)_{n\in\mathbb{N}}$ such that the pair $((h_n)_{n\in\mathbb{N}},\eta)$ is universally converging.
\qed
\end{lemma}

\noindent Given two projective currents $\eta_1,\eta_2\in\mathbb{P}Curr_N$, we define

\begin{displaymath}
E(\eta_1,\eta_2):=\{T\in\widetilde{\mathcal{FI}}|\langle T,\eta_1\rangle\neq 0 \text{~and~}\langle T,\eta_2\rangle\neq 0\}.
\end{displaymath} 

\begin{lemma} \label{E-measurable}
Let $\mu$ be a probability measure on $\text{Out}(F_N)$, whose support generates a nonelementary subgroup of $\text{Out}(F_N)$. For all $\eta_1,\eta_2\in\mathbb{P}Curr_N$ and all $\mu$-stationary measures $\nu$ on $\partial CV_N$, the set $E(\eta_1,\eta_2)$ is measurable and has full $\nu$-measure.
\end{lemma}

\begin{proof}
Measurability of $E(\eta_1,\eta_2)$ comes from measurability of $\widetilde{\mathcal{FI}}$ and continuity of the intersection form (Theorem \ref{intersection-form}). By Proposition \ref{stationary-on-outer-space}, we have $\nu(\widetilde{\mathcal{FI}})=1$. Theorem \ref{arational-duality} implies that $\widetilde{\mathcal{FI}}\smallsetminus E(\eta_1,\eta_2)$ consists of at most two classes of the relation $\sim$. As each of these classes has $\nu$-measure $0$ (Proposition \ref{nonatomic}), we get that $\nu(E(\eta_1,\eta_2))=1$. 
\end{proof}

\begin{lemma} \label{limit-points}
Let $((h_n)_{n\in\mathbb{N}},\eta)$ be a universally converging pair, and $T'\in CV_N\cup E(\eta^{\infty},\eta_0^{\infty})$. If $(h_nT')_{n\in\mathbb{N}}$ converges to a tree $T'_{\infty}\in\overline{CV_N}$, then $\langle T'_{\infty},\eta\rangle=0$.
\end{lemma}

\begin{proof}
Let $T_{\infty}$ be the limit of the sequence $(h_n\ast_{CV_N})_{n\in\mathbb{N}}$. We choose lifts of $T_{\infty}$, $T'_{\infty}$ and $T'$ to the unprojectivized outer space $\overline{cv_N}$, which we again denote by $T_{\infty}$, $T'_{\infty}$ and $T'$, slightly abusing notations, and we denote by $\ast_{cv_N}$ a lift of $\ast_{CV_N}$ to $cv_N$. There exist sequences $(t_n)_{n\in\mathbb{N}}$ and $(t'_n)_{n\in\mathbb{N}}$ of positive real numbers such that $t_nh_n\ast_{cv_N}$ converges to $T_{\infty}$, and $t'_nh_nT'$ converges to $T'_{\infty}$. Similarly, as $((h_n)_{n\in\mathbb{N}},\eta)$ is universally converging, there exist sequences $(\lambda_n^0)_{n\in\mathbb{N}}$ and $(\lambda_n)_{n\in\mathbb{N}}$ such that $\lambda_n^0h_n^{-1}\eta_0$ converges (non-projectively) to $\eta_0^{\infty}$, and $\lambda_nh_n^{-1}\eta$ converges (non-projectively) to $\eta^{\infty}$. We have $\langle \ast_{cv_N},\eta_0^{\infty}\rangle\neq 0$ (Theorem \ref{currents-duality}). By continuity and $\text{Out}(F_N)$-invariance of the intersection form (Theorem \ref{intersection-form}), we have

\begin{displaymath}
\frac{\langle T_{\infty},\eta_0\rangle}{\langle \ast_{cv_N},\eta_0^{\infty}\rangle}=\frac{\lim_{n\to +\infty} t_n\langle h_n\ast_{cv_N},\eta_0\rangle}{\lim_{n\to +\infty} \lambda_n^{0}\langle \ast_{cv_N},h_n^{-1}\eta_0\rangle}=\lim_{n\to +\infty}\frac{t_n}{\lambda_n^{0}},
\end{displaymath}

\noindent and similarly, as $T'\in CV_N\cup E(\eta^{\infty},\eta_0^{\infty})$, we have 

\begin{displaymath}
\frac{\langle T'_{\infty},\eta_0\rangle}{\langle T',\eta_0^{\infty}\rangle}=\frac{\lim_{n\to +\infty} t'_n\langle h_nT',\eta_0\rangle}{\lim_{n\to +\infty} \lambda_n^{0}\langle T',h_n^{-1}\eta_0\rangle}=\lim_{n\to +\infty}\frac{t'_n}{\lambda_n^{0}}.
\end{displaymath}

We have $\langle T_{\infty},\eta_0\rangle>0$ and $\langle T'_{\infty},\eta_0\rangle>0$, so the limits $\lim_{n\to +\infty}\frac{t_n}{\lambda_n^{0}}$ and $\lim_{n\to +\infty}\frac{t'_n}{\lambda_n^{0}}$ are both finite and non-zero, and hence there exists a non-zero limit $\lim_{n\to +\infty}\frac{t_n}{t'_n}$. As $\langle \ast_{cv_N},\eta^{\infty}\rangle\neq 0$ and $T'\in CV_N\cup E(\eta^{\infty},\eta_0^{\infty})$, the same argument as above also shows that

\begin{displaymath}
0=\frac{\langle T_{\infty},\eta\rangle}{\langle \ast_{cv_N},\eta^{\infty}\rangle}=\lim_{n\to +\infty}\frac{t_n}{\lambda_n},
\end{displaymath}

\noindent and

\begin{displaymath}
\frac{\langle T'_{\infty},\eta\rangle}{\langle T',\eta^{\infty}\rangle}=\lim_{n\to +\infty}\frac{t'_n}{\lambda_n},
\end{displaymath}

\noindent thus proving that $\langle T'_{\infty},\eta\rangle=0$.
\end{proof}

\begin{proof}[Proof of Proposition \ref{limit-points-2}]
As $(g_n)_{n\in\mathbb{N}}$ is unbounded, the sequence $(g_n\ast_{CV_N})_{n\in\mathbb{N}}$ has a limit point  $T_{\infty}\in\partial CV_N$. Theorem \ref{currents-duality} and Lemma \ref{universal-subsequence} provide a current $\eta\in\text{Dual}(T_{\infty})$, and a subsequence $(h_n)_{n\in\mathbb{N}}$ of $(g_n)_{n\in\mathbb{N}}$ such that the pair $((h_n)_{n\in\mathbb{N}},\eta)$ is universally converging. By Lemma \ref{E-measurable}, the set $E(\eta^{\infty},\eta_0^{\infty})$ is measurable and $\nu(E(\eta^{\infty},\eta_0^{\infty}))=1$, and by Lemma \ref{limit-points}, all limit points of sequences $(h_nT')_{n\in\mathbb{N}}$ for $T'\in E(\eta^{\infty},\eta_0^{\infty})$ belong to the closed set $\widetilde{\text{Dual}}(\eta):=\{T\in\overline{CV_N}|\langle T,\eta\rangle=0\}$. Lemma \ref{KM-limit-points} shows that $\lambda$ is concentrated on $\widetilde{\text{Dual}}(\eta)$. If $T_{\infty}\in\widetilde{\mathcal{FI}}$, Theorem \ref{arational-duality} implies that $\widetilde{\text{Dual}}(\eta)$ is contained in $\widetilde{\mathcal{FI}}$, in a single class of the relation $\sim$. If $T_{\infty}\in\partial CV_N\smallsetminus\widetilde{\mathcal{FI}}$, then Theorem \ref{arational-duality} implies that for all $T\in\widetilde{\mathcal{FI}}$, we have $\langle T,\eta\rangle\neq 0$. Hence $\widetilde{\text{Dual}}(\eta)\subseteq\partial CV_N\smallsetminus\widetilde{\mathcal{FI}}$, so $\lambda$ is concentrated on $\partial CV_N\smallsetminus\widetilde{\mathcal{FI}}$. 

Now let $T\in CV_N$, and let $T'_{\infty}$ be a limit point of the sequence $(g_nT)_{n\in\mathbb{N}}$. In other words, there exists an unbounded subsequence $(h'_n)_{n\in\mathbb{N}}$ of $(g_n)_{n\in\mathbb{N}}$ such that the sequence $(h'_nT)_{n\in\mathbb{N}}$ converges to $T'_{\infty}$, and up to passing to a subsequence again, we may assume that the sequence $(h'_n\ast_{CV_N})_{n\in\mathbb{N}}$ converges to a tree $T''_{\infty}\in\partial CV_N$. Notice that $T''_{\infty}\in \widetilde{\mathcal{FI}}$ if and only if $T_{\infty}\in \widetilde{\mathcal{FI}}$, and in this case they belong to the same $\sim$-class, otherwise the above argument applied to both $T_{\infty}$ and $T''_{\infty}$ would imply that $\lambda$ is simultaneously supported on two disjoint measurable sets. The last part of the claim then follows from Lemma \ref{limit-points}.
\end{proof}

\subsection{Uniqueness of the stationary measure on $\mathcal{FI}$}

Let $\mu$ be a probability measure on $\text{Out}(F_N)$, whose support generates a nonelementary subgroup of $\text{Out}(F_N)$. Given a sample path $(g_n)_{n\in\mathbb{N}}$ of the random walk on $(\text{Out}(F_N),\mu)$, we denote by $\xi(\mathbf{g})\in\mathcal{FI}$ the limit of any sequence $(g_n T_0)_{n\in\mathbb{N}}$ (with $T_0\in CV_N$), which $\mathbb{P}$-almost surely exists and is independent from $T_0$ by Theorem \ref{convergence-drift}. The \emph{hitting measure} $\nu$ on $\mathcal{FI}$ is the $\mu$-stationary measure defined by letting

\begin{displaymath}
\nu(X):=\mathbb{P}(\xi(\mathbf{g})\in X)
\end{displaymath}

\noindent for all Borel subsets $X\subseteq\mathcal{FI}$.

\begin{prop}\label{unique}
Let $\mu$ be a probability measure on $\text{Out}(F_N)$, whose support generates a nonelementary subgroup of $\text{Out}(F_N)$. Then the hitting measure is the unique $\mu$-stationary measure on $\mathcal{FI}$.
\end{prop}

\begin{proof}
Let $\nu$ be a $\mu$-stationary measure on $\mathcal{FI}$. For $\mathbf{g}\in G^{\mathbb{N}}$, let $\lambda(\mathbf{g})$ be the Dirac measure on $\xi(\mathbf{g})$. As $\nu$ is purely nonatomic (Proposition \ref{existence-stationary-ff}), Lemma \ref{limit-points} shows that for $\mathbb{P}$-a.e. sample path $\mathbf{g}\in G^{\mathbb{N}}$ of the random walk, and $\nu$-a.e. $x\in\mathcal{FI}$, the sequence $(g_nx)_{n\in\mathbb{N}}$ converges to $\xi(\mathbf{g})$. So for all bounded continuous functions $F$ on $\mathcal{FI}$, the integrals 

\begin{displaymath}
\int_{\mathcal{FI}}F(g_nx)d\nu(x)
\end{displaymath}

\noindent converge to $F(\xi(\mathbf{g}))$ as $n$ goes to $+\infty$, thus showing that the measures $g_n\nu$ weakly converge to $\lambda(\mathbf{g})$. In other words, for $\mathbb{P}$-a.e. sample path $\mathbf{g}\in G^{\mathbb{N}}$ of the random walk, and all open subsets $U\subseteq\mathcal{FI}$, we have

\begin{displaymath}
\liminf_{n\to +\infty}g_n\nu(U)\ge \lambda(\mathbf{g})(U)
\end{displaymath}

\noindent by the Portmanteau Theorem (see \cite[Theorem 2.1]{Bil68}). For all $n\in\mathbb{N}$, the measure $\nu$ is $\mu^{\ast n}$-stationary, so for all $n\in\mathbb{N}$ and all open subsets $U\subseteq\mathcal{FI}$, we have

\begin{displaymath}
\nu(U)=\int_{G^{\mathbb{N}}}g_n\nu(U)d\mathbb{P}(\mathbf{g}).
\end{displaymath} 

\noindent  Hence 

\begin{displaymath}
\nu(U)\ge\left(\int_{G^{\mathbb{N}}}\lambda(\mathbf{g})d\mathbb{P}(\mathbf{g})\right)(U),
\end{displaymath}

\noindent and $\int_{G^{\mathbb{N}}}\lambda(\mathbf{g})d\mathbb{P}(\mathbf{g})$ is the hitting measure on $\mathcal{FI}$. Regularity of $\nu$ and of the hitting measure \cite[Theorem 1.1]{Bil68} implies that the inequality holds true for all Borel subsets of $\mathcal{FI}$. As both $\nu$ and the hitting measure are probability measures on $\mathcal{FI}$, they are equal.
\end{proof}

\section{The Poisson boundary of $\text{Out}(F_N)$} \label{sec-Poisson}

This section is devoted to the description of the Poisson boundary of $\text{Out}(F_N)$. We start by recalling the construction of the Poisson boundary of a finitely generated group equipped with a probability measure.

\subsection{Generalities on Poisson boundaries and statement of the main result}\label{sec-Poisson-general}

Let $G$ be a finitely generated group, and $\mu$ be a probability measure on $G$. The \emph{Poisson boundary} of $(G,\mu)$ is the space of ergodic components of the \emph{time shift} $T$, defined in the path space of the random walk on $(G,\mu)$ by $(T\textbf{g})_n=g_{n+1}$. More precisely, let $\mathcal{A}_T$ be the $\sigma$-algebra of all $T$-invariant measurable subsets of the path space $G^{\mathbb{N}}$, and let $\overline{\mathcal{A}_T}$ be its completion with respect to the measure $\mathbb{P}_m=\sum_{g\in G}g\mathbb{P}$ corresponding to the distribution of the sample paths of a random walk whose initial distribution is the counting measure on $G$. We recall that $\mathcal{A}$ denotes the $\sigma$-algebra on the path space $G^{\mathbb{N}}$ generated by the cylinder subsets. Let $\overline{\mathcal{A}}$ be its completion with respect to the measure $\mathbb{P}_m$. Since $(G^{\mathbb{N}},\overline{\mathcal{A}},\mathbb{P}_m)$ is a Lebesgue space, the Rokhlin correspondence associates to $\overline{\mathcal{A}_T}$ a measurable partition $\eta$ of $G^{\mathbb{N}}$, see \cite{Rok49} (we recall that a partition of a measurable space into measurable subsets is \emph{measurable} if it is countably separated). This partition is unique in the sense that if $\eta$ and $\eta'$ are two such partitions, then there exists a subset of $G^{\mathbb{N}}$ of full $\mathbb{P}_m$-measure on which they coincide. We call it the \emph{Poisson partition} of $G^{\mathbb{N}}$. The quotient space $\Gamma:=(G^{\mathbb{N}},\overline{\mathcal{A}})/\eta$ carries several measures. On the one hand, it can be equipped with the image $\nu_m$ of the measure $\mathbb{P}_m$ under the quotient map, and $(\Gamma,\nu_m)$ is a Lebesgue space. On the other hand, it can be equipped with the \emph{harmonic measure} $\nu$, which is the image of $\mathbb{P}$ under the quotient map. The measure $\nu$ is $\mu$-stationary, and the measure $\nu_m$ can be recovered from $\nu$ by the formula 

\begin{displaymath}
\nu_m=\sum_{g\in G}g\nu.
\end{displaymath}

\noindent We call $(\Gamma,\nu)$ the \emph{Poisson boundary} of $(G,\mu)$.

A \emph{$\mu$-boundary} is a probability space $(B,\lambda)$, which is the quotient of the path space $(G^{\mathbb{N}},\mathbb{P})$ with respect to some shift-invariant and $G$-invariant measurable partition. Equivalently, a $\mu$-boundary is a probability space which is the quotient of the Poisson boundary with respect to some $G$-invariant measurable partition. So the Poisson boundary is itself a $\mu$-boundary, and it is maximal with respect to this property. Typical examples of $\mu$-boundaries arise when $G$ is embedded into a metric separable $G$-space, and $\mathbb{P}$-a.e. sample path $\textbf{g}$ converges to a limit $bnd(\textbf{g})$.

In \cite{Kai00}, Kaimanovich gave a criterion for checking that a $\mu$-boundary is maximal. Let $d$ be the word metric on $G$ with respect to some finite generating set -- any two such metrics are bi-Lipschitz equivalent. The \emph{first logarithmic moment} of $\mu$ with respect to $d$ is defined (with the convention that $\log 0=0$) as

\begin{displaymath}
|\mu|:=\sum_{g\in G}\log d(e,g)\mu(g).
\end{displaymath}

\noindent The \emph{entropy} of $\mu$ is defined as

\begin{displaymath}
H(\mu):=\sum_{g\in G}-\mu(g)\log \mu(g).
\end{displaymath}

\noindent Given a measure $\mu$ on a countable group $G$, we denote by $\check{\mu}$ the \emph{reflected measure} on $G$ defined by $\check{\mu}(g):=\mu(g^{-1})$ for all $g\in G$. In the following statement, we take the convention that $\log 0=0$. 

\begin{theo} \label{Kaimanovich-criterion} (Kaimanovich \cite[Theorem 6.5]{Kai00})
Let $G$ be a finitely generated group, let $d$ be a word metric on $G$, and let $\mu$ be a probability measure on $G$ which has finite first logarithmic moment with respect to $d$, and finite entropy. Let $(B_-,\nu_-)$ and $(B_+,\nu_+)$ be $\check{\mu}$- and $\mu$-boundaries, respectively, and assume there exists a measurable $G$-equivariant map

\begin{displaymath}
\begin{array}{cccc}
B_-\times B_+&\to & 2^G\\
(b_-,b_+)&\mapsto & S(b_-,b_+)
\end{array}
\end{displaymath} 
 
\noindent such that for $\nu_-\otimes\nu_+$-a.e. $(b_-,b_+)\in B_-\times B_+$, the set $S(b_-,b_+)$ is nonempty, and

\begin{displaymath}
 \sup_{k\in\mathbb{N}\smallsetminus\{0\}}\frac{1}{\log k}\log\text{card} [S(b_-,b_+)\cap \mathcal{B}_k]<+\infty,
\end{displaymath} 
 
\noindent where $\mathcal{B}_k$ denotes the $d$-ball of radius $k$ centered at $e$. Then the boundaries $(B_-,\nu_-)$ and $(B_+,\nu_+)$ are Poisson boundaries.
\end{theo}

In terms of $\mu$-boundaries, Theorem \ref{convergence-drift} can be restated as follows.

\begin{theo}\label{mu-boundary}
Let $\mu$ be a probability measure on $\text{Out}(F_N)$, whose support generates a nonelementary subgroup of $\text{Out}(F_N)$, and let $\nu$ be the hitting measure on $\mathcal{FI}$. Then $(\mathcal{FI},\nu)$ is a $\mu$-boundary.
\end{theo}

\begin{proof}
Theorem \ref{convergence-drift} provides an (almost-surely well-defined) measurable map 

\begin{displaymath}
bnd:\text{Out}(F_N)^{\mathbb{N}}\to\mathcal{FI},
\end{displaymath}

\noindent that sends a sample path $(g_n)_{n\in\mathbb{N}}$ to the limit of the sequence $(g_nT_0)_{n\in\mathbb{N}}$ for any $T_0\in CV_N$. The space $\mathcal{FI}$ is metrizable \cite[Corollary 7.2]{BR13} and separable (it is a quotient of a subspace of a separable metric space), so its Borel $\sigma$-algebra is countably separated. This implies that the $bnd$-preimage of the point partition of $\mathcal{FI}$ is a measurable partition of the path space $G^{\mathbb{N}}$, and therefore $(\mathcal{FI},\nu)$ is a $\mu$-boundary. 
\end{proof}

Under some further hypotheses on the measure $\mu$, we will show that the $\mu$-boundary $(\mathcal{FI},\nu)$ is the Poisson boundary of $(\text{Out}(F_N),\mu)$.

\begin{theo} \label{Poisson}
Let $\mu$ be a probability measure on $\text{Out}(F_N)$ such that $gr(\mu)$ is nonelementary, which has finite first logarithmic moment with respect to the word metric, and finite entropy. Let $\nu$ be the hitting measure on $\mathcal{FI}$. Then the measure space $(\mathcal{FI},\nu)$ is the Poisson boundary of $(\text{Out}(F_N),\mu)$.
\end{theo} 

We will use Kaimanovich's criterion (Theorem \ref{Kaimanovich-criterion}) to prove the maximality of the $\mu$-boundary provided by Theorem \ref{mu-boundary}. Our construction of the strips is inspired from Hamenstädt's construction of \emph{lines of minima} in outer space \cite{Ham12-2}. 

\subsection{Axes in outer space}\label{sec-construction}

The following construction is inspired from Hamenstädt's construction of \emph{lines of minima} \cite{Ham12-2}. We recall from Section \ref{sec-arational} that a pair $([\eta],[\eta'])\in\mathbb{P}Curr_N^2$ is \emph{positive} if $\langle T,\eta+\eta'\rangle>0$ for all $T\in\overline{CV_N}$. Given a positive pair of projective currents $([\eta],[\eta'])\in\mathbb{P}Curr_N^2$, we want to define an axis in outer space which will roughly consist of trees for which either $\eta$ or $\eta'$ can serve as a fairly good candidate for computing the infimal Lipschitz distortion to any other tree in the closure of outer space. We first define 

\begin{displaymath}
\begin{array}{cccc}
l_{[\eta],[\eta']}:&CV_N\times\overline{CV_N}&\to &\mathbb{R}\\
& (T,T')&\mapsto &\frac{\text{Lip}(T,T')}{\Lambda_{\{[\eta],[\eta']\}}(T,T')},
\end{array}
\end{displaymath}

\noindent where we recall the notations from Section \ref{sec-metric}. This measures to which extent the stretch of either $\eta$ or $\eta'$ gives a good estimate of the Lipschitz distortion $\text{Lip}(T,T')$. We always have $l_{[\eta],[\eta']}(T,T')\ge 1$ (the closer to $1$ it is, the better the estimate will be), and positivity of the pair $([\eta],[\eta'])$ ensures that $l_{[\eta],[\eta']}(T,T')<+\infty$. Notice that this only depends on the projective classes of the trees $T$ and $T'$ and the currents $[\eta]$ and $[\eta']$. So for all $T\in CV_N$, the map $l_{[\eta],[\eta']}(T,.)$ is a continuous function on a compact set, so we can let

\begin{displaymath}
L_{[\eta],[\eta']}(T):=\max_{T'\in\overline{CV_N}}l_{[\eta],[\eta']}(T,T')<+\infty.
\end{displaymath}

Given $L\ge 1$, a tree $T\in CV_N$ satisfies $L_{[\eta],[\eta']}(T)\le L$ if for all $T'\in\overline{CV_N}$, the stretch of either $\eta$ or $\eta'$ from $T$ to $T'$ gives a good estimate of the Lipschitz distortion $\text{Lip}(T,T')$, up to an error controlled by $L$. We define the \emph{$L$-axis} $A_L([\eta],[\eta'])$ of a positive pair of projective currents as the set of all $T\in CV_N$ such that $L_{[\eta],[\eta']}(T)\le L$. For all $\Psi\in\text{Out}(F_N)$, all positive pairs $([\eta],[\eta'])\in\mathbb{P}Curr_N^2$ and all $T\in CV_N$, we have $L_{\Psi[\eta],\Psi[\eta']}(\Psi T)=L_{[\eta],[\eta']}(T)$, so the $L$-axis $A_L([\eta],[\eta'])$ depends $\text{Out}(F_N)$-equivariantly on the positive pair $([\eta],[\eta'])\in\mathbb{P}Curr_N^2$.

We now associate to any pair $(T_-,T_+)\in\mathcal{FI}\times\mathcal{FI}\smallsetminus\Delta$ (where $\Delta$ denotes the diagonal) an axis in $CV_N$. The key point is that associated to any free and arational tree is a finite set of ergodic currents (Proposition \ref{dual-arational}), and given $(T_-,T_+)\in\mathcal{FI}\times\mathcal{FI}\smallsetminus\Delta$, any pair of currents $([\eta_-],[\eta_+])\in\text{Erg}(T_-)\times\text{Erg}(T_+)$ is positive (Corollary \ref{arational-current}). Given $L\ge 1$, we define the \emph{$L$-axis} $A_L(T_-,T_+)$ as the union of all $A_L([\eta_-],[\eta_+])$, with $([\eta_-],[\eta_+])$ varying in the finite set $\text{Erg}(T_-)\times\text{Erg}(T_+)$. For $T\in CV_N$, letting

\begin{displaymath}
L_{T_-,T_+}(T):=\min_{\substack{[\eta_-]\in \text{Erg}(T_-)\\\ [\eta_+]\in \text{Erg}(T_+)}}L_{[\eta_-],[\eta_+]}(T),
\end{displaymath} 

\noindent the $L$-axis $A_L(T_-,T_+)$ is also equal to the set of all $T\in CV_N$ such that $L_{T_-,T_+}(T)\le L$. 

\begin{rk}
Hamenstädt has shown in \cite[Proposition 4.9]{Ham12-2} that all accumulation points of the $L$-axis of a pair $(T_-,T_+)\in\mathcal{FI}\times\mathcal{FI}\smallsetminus\Delta$ are free and arational, and project to either $T_-$ or $T_+$ in $\mathcal{FI}$.
\end{rk}

\subsection{Definition of the strips}\label{sec-construction-2}

In order to use Kaimanovich's criterion for proving Theorem \ref{Poisson}, we need to associate to every pair $(T_-,T_+)\in\mathcal{FI}\times\mathcal{FI}$ a strip in $\text{Out}(F_N)$. We fix once and for all a basepoint $\ast_{CV_N}\in CV_N$. 

\begin{de}
Let $(T_-,T_+)\in\mathcal{FI}\times\mathcal{FI}$, and let $L\ge 1$. If $T_-\neq T_+$, the \emph{$L$-strip} $S_L(T_-,T_+)$ is defined to be the set of all $\Phi\in\text{Out}(F_N)$ such that $\Phi \ast_{CV_N}\in A_L(T_-,T_+)$. If $T_-=T_+$, we let $S_L(T_-,T_+)=\emptyset$.
\end{de}

For all $\Psi\in\text{Out}(F_N)$ and all $L\ge 1$, we have $S_L(\Psi T_-,\Psi T_+)=\Psi S_L(T_-,T_+)$. In view of Theorems \ref{Kaimanovich-criterion} and \ref{mu-boundary}, Theorem \ref{Poisson} will be a consequence of the following three facts.

\begin{prop}\label{nonemptiness}
There exists $L_1>1$ such that for $\nu_-\otimes\nu_+$-a.e. $(T_-,T_+)\in\mathcal{FI}\times\mathcal{FI}$, we have $S_{L_1}(T_-,T_+)\neq\emptyset$.
\end{prop} 

In the next two statements, we fix the constant $L_1$ provided by Proposition \ref{nonemptiness}. For all $k\in\mathbb{N}$, let $\mathcal{B}_k$ be the ball of radius $k$ in $\text{Out}(F_N)$ for the word metric. 

\begin{prop}\label{strips}
For $\nu_-\otimes\nu_+$-a.e. $(T_-,T_+)\in \mathcal{FI}\times\mathcal{FI}$, there exists $\lambda\in\mathbb{R}$ such that for all $k\in\mathbb{N}$, we have

\begin{displaymath}
 \text{card}(S_{L_1}(T_-,T_+)\cap \mathcal{B}_k)\le \lambda k.
 \end{displaymath}
\end{prop}

\begin{prop}\label{measurability}
The map

\begin{displaymath}
\begin{array}{cccc}
\mathcal{FI}\times\mathcal{FI} &\to &2^{\text{Out}(F_N)}\\
(T_-,T_+) &\mapsto & S_{L_1}(T_-,T_+)
\end{array}
\end{displaymath}

\noindent is measurable.
\end{prop}

\subsection{Choosing $L_1$ to ensure nonemptiness of the strips}\label{sec-nonempty}

Our proof of Proposition \ref{nonemptiness} is inspired from Kaimanovich and Masur's analogous argument in the case of mapping class groups of surfaces \cite[Theorem 2.3.1]{KM96}.
 
\begin{proof}[Proof of Proposition \ref{nonemptiness}]
Consider the measure space $(\text{Out}(F_N)^{\mathbb{Z}},\overline{\mathbb{P}})$ of \emph{bilateral paths} $\overline{\textbf{g}}=(g_n)_{n\in\mathbb{Z}}$ satisfying $g_0=e$, corresponding to bilateral sequences of independent $\mu$-distributed increments $(s_n)_{n\in\mathbb{Z}}$ by the formula $g_n=g_{n-1}s_n$. The unilateral paths $\textbf{g}=(g_n)_{n\ge 0}$ and $\check{\textbf{g}}=(g_{-n})_{n\ge 0}$ are independent, and correspond to sample paths of the random walks on $(\text{Out}(F_N),\mu)$ and $(\text{Out}(F_N),\check{\mu})$, respectively. The \emph{Bernoulli shift} $U$ is the transformation defined in the space $(\text{Out}(F_N)^{\mathbb{Z}},\mu^{\otimes\mathbb{Z}})$ of increments $s=(s_n)_{n\in\mathbb{Z}}$ by $(Us)_n=s_{n+1}$ for all $n\in\mathbb{Z}$. We again denote by $U$ the measure-preserving, ergodic transformation induced by the Bernoulli shift in the space of bilateral paths $(\text{Out}(F_N)^{\mathbb{Z}},\overline{\mathbb{P}})$, defined by 

\begin{displaymath}
(U^k\overline{\textbf{g}})_n=g_k^{-1}g_{n+k}.
\end{displaymath}

Let $(\mathcal{FI},\nu_-)$ and $(\mathcal{FI},\nu_+)$ be the boundaries corresponding to $(\text{Out}(F_N),\check{\mu})$ and $(\text{Out}(F_N),\mu)$ provided by Theorem \ref{mu-boundary}. We let $bnd_-(\overline{\textbf{g}})\in \mathcal{FI}$ (resp. $bnd_+(\overline{\textbf{g}})\in \mathcal{FI}$) be the limit as $n$ goes to $+\infty$ of the sequence $(g_{-n}\ast_{CV_N})_{n\in\mathbb{N}}$ (resp. $(g_n\ast_{CV_N})_{n\in\mathbb{N}}$), which is $\overline{\mathbb{P}}$-almost surely well-defined by Theorem \ref{convergence-drift}. Then for all $k\in\mathbb{Z}$, we have
 
\begin{displaymath}
bnd_-(U^k\overline{\textbf{g}})=g_k^{-1}bnd_-(\overline{\textbf{g}}),
\end{displaymath}
  
\noindent and similarly

\begin{displaymath}
bnd_+(U^k\overline{\textbf{g}})=g_k^{-1}bnd_+(\overline{\textbf{g}}).
\end{displaymath}

\noindent Let

\begin{displaymath} 
\psi(\overline{\textbf{g}}):=L_{bnd_-(\overline{\textbf{g}}),bnd_+(\overline{\textbf{g}})}(\ast_{CV_N})
\end{displaymath}  
  
\noindent (when $bnd_-(\overline{\textbf{g}})=bnd_+(\overline{\textbf{g}})$, we let $\psi(\overline{\mathbf{g}}):=+\infty$). Measurability of $\psi$ will follow from the proof of Proposition \ref{measurability} in Section \ref{sec-measurable}. For $\overline{\mathbb{P}}$-a.e $\overline{\mathbf{g}}:=(g_n)_{n\in\mathbb{Z}}$, we have $\psi(\overline{\textbf{g}})<+\infty$. Hence there exists $L_1>1$ such that $\overline{\mathbb{P}}[\psi(\overline{\textbf{g}})\le L_1]>0$. For all $k\in\mathbb{N}$, we have 

\begin{displaymath}
\begin{array}{rl}
\psi(U^k\overline{\textbf{g}})&=L_{g_k^{-1}bnd_-(\overline{\textbf{g}}),g_k^{-1}bnd_+(\overline{\textbf{g}})}(\ast_{CV_N})\\
&=L_{bnd_-(\overline{\textbf{g}}),bnd_+(\overline{\textbf{g}})}(g_k\ast_{CV_N}).
\end{array}
\end{displaymath}

\noindent Applying Birkhoff's ergodic theorem \cite{Bir31} to the ergodic transformation $U$, we get that for $\overline{\mathbb{P}}$-a.e. bilateral path $\overline{\textbf{g}}$, the density of times $k\ge 0$ such that $L_{bnd_-(\overline{\textbf{g}}),bnd_+(\overline{\textbf{g}})}(g_k\ast_{CV_N})\le L_1$ is positive. Therefore, for $\overline{\mathbb{P}}$-a.e. bilateral path $\overline{\mathbf{g}}$, the $L_1$-strip $S_{L_1}(bnd_-(\overline{\textbf{g}}),bnd_+(\overline{\textbf{g}}))$ is nonempty, so $\nu_-\otimes\nu_+$-a.e. the set $S_{L_1}(T_-,T_+)$ is nonempty.  
\end{proof}

\begin{rk}
The proof of Proposition \ref{nonemptiness} actually shows that for $\overline{\mathbb{P}}$-a.e. bilateral path $\overline{\mathbf{g}}:=(g_n)_{n\in\mathbb{Z}}$ of the random walk on $(\text{Out}(F_N),\mu)$, the density of times $k\in\mathbb{N}$ such that $g_k\in S_{L_1}(bnd_-(\overline{\mathbf{g}}),bnd_+(\overline{\mathbf{g}}))$ is positive.
\end{rk}

\subsection{Thinness of the strips} \label{sec-linear-growth}

In this head, we will prove Proposition \ref{strips}. Our argument is inspired from Hamenstädt's estimates in \cite[Proposition 4.4]{Ham12-2}. From now on, we fix a (non-projective) positive pair $(\eta_-,\eta_+)\in Curr_N^2$. Given $T\in CV_N$, we let $\sigma(T)\in\mathbb{R}$ be such that $\langle T,\eta_+\rangle=e^{\sigma(T)}\langle T,\eta_-\rangle$. This defines a "height function" on the axis $A_{L_1}([\eta_-],[\eta_+])$. We will show that $A_{L_1}([\eta_-],[\eta_+])$ is close to being a $d_{sym}$-geodesic with holes. Proposition \ref{strips} will then follow from proper discontinuity of the action of $\text{Out}(F_N)$ on $CV_N$.

\begin{prop}\label{axes-almost-geodesics}
For all $(\eta_-,\eta_+)\in Curr_N^2$, and all $S,T\in A_{L_1}([\eta_-],[\eta_+])$, we have

\begin{displaymath}
|\sigma(S)-\sigma(T)|\le d_{sym}(S,T)\le |\sigma(S)-\sigma(T)|+2\log L_1.
\end{displaymath}
\end{prop}

\begin{proof}
We have $\Lambda_{[\eta_+]}(S,T)=e^{\sigma(T)-\sigma(S)}\Lambda_{[\eta_-]}(S,T)$. Assume without loss of generality that $\sigma(S)\le \sigma(T)$. Then $\Lambda_{\{[\eta_-],[\eta_+]\}}(S,T)=\Lambda_{[\eta_+]}(S,T)$, and we get from the definition of $A_{L_1}([\eta_-],[\eta_+])$ that

\begin{displaymath}
\frac{1}{L_1}\text{Lip}(S,T)\le \Lambda_{[\eta_+]}(S,T)\le \text{Lip}(S,T).
\end{displaymath}

\noindent Taking logarithms, we get

\begin{displaymath}
d(S,T)-\log L_1\le \log\Lambda_{[\eta_+]}(S,T)\le d(S,T).
\end{displaymath} 

\noindent Reversing the roles of $S$ and $T$, we also have

\begin{displaymath}
d(T,S)-\log L_1\le \log\Lambda_{[\eta_-]}(T,S)\le d(T,S).
\end{displaymath} 

\noindent By summing the above inequalities, we obtain
\begin{displaymath}
d_{sym}(S,T)-2\log L_1\le \sigma(T)-\sigma(S)\le d_{sym}(S,T),
\end{displaymath}

\noindent which is the desired inequality.
\end{proof}

\begin{proof}[Proof of Proposition \ref{strips}]
Let $T_-\neq T_+\in\mathcal{FI}$. As $\text{Erg}(T_-)$ and $\text{Erg}(T_+)$ are finite, it is enough to show that for all pairs $([\eta_-],[\eta_+])\in\text{Erg}(T_-)\times\text{Erg}(T_+)$, the cardinality of the set $\{\Phi\in\mathcal{B}_k|\Phi \ast_{CV_N} \in A_{L_1}([\eta_-],[\eta_+])\}$ grows linearly with $k$. For all $k\in\mathbb{N}$, we denote by $\mathcal{B}^{sym}_k$ the $d_{sym}$-ball centered at $\ast_{CV_N}$ in $CV_N$. There exists $C\in\mathbb{R}$ such that for all $k\in\mathbb{N}$, and all $\Phi\in\mathcal{B}_k$, we have $\Phi \ast_{CV_N}\in\mathcal{B}^{sym}_{Ck}$. Therefore, it is enough to check that for all pairs $([\eta_-],[\eta_+])\in\text{Erg}(T_-)\times\text{Erg}(T_+)$, the cardinality of the set $\{\Phi\in\text{Out}(F_N)|\Phi\ast_{CV_N}\in A_{L_1}([\eta_-],[\eta_+])\cap\mathcal{B}^{sym}_k\}$ grows linearly with $k$.

Let $([\eta_-],[\eta_+])\in\text{Erg}(T_-)\times\text{Erg}(T_+)$. We fix a representative $\eta_-$ (resp. $\eta_+$) of $[\eta_-]$ (resp. $[\eta_+]$) in $Curr_N$. For all $T\in A_{L_1}([\eta_-],[\eta_+])$, let $f(T):=\lfloor \sigma(T)\rfloor$. Denote by $M$ the maximal cardinality of the intersection of a $d_{sym}$-ball of radius $1+2\log L_1$ with the $\text{Out}(F_N)$-orbit of $\ast_{CV_N}$ (which is finite by proper discontinuity of the action). Proposition \ref{axes-almost-geodesics} shows that 

\begin{itemize}
\item the $f$-preimage of any integer has diameter at most $1+2\log L_1$, so its intersection with the $\text{Out}(F_N)$-orbit of $\ast_{CV_N}$ has cardinality at most $M$, and
\item for all $\Phi,\Psi\in\text{Out}(F_N)$, if $\Phi\ast_{CV_N}$ and $\Psi\ast_{CV_N}$ both belong to $A_{L_1}([\eta_-],[\eta_+])\cap\mathcal{B}_k^{sym}$, then $|f(\Phi\ast_{CV_N})-f(\Psi\ast_{CV_N})|\le 2k$.
\end{itemize}

The cardinality of $\{\Phi\in\text{Out}(F_N)|\Phi\ast_{CV_N}\in A_{L_1}([\eta_-],[\eta_+])\cap\mathcal{B}^{sym}_k\}$ is therefore bounded above by $(2k+1)M$. 
\end{proof}

\subsection{Measurable dependence of the strips on the pair $(T_-,T_+)\in\mathcal{FI}\times\mathcal{FI}$}\label{sec-measurable}

\begin{proof}[Proof of Proposition \ref{measurability}]
Since $\text{Out}(F_N)$ is countable, we only need to check that for all $\Phi\in\text{Out}(F_N)$, the set

\begin{displaymath}
S_{L_1}^{-1}(\Phi):=\{(T_-,T_+)\in \mathcal{FI}\times\mathcal{FI}|\Phi\in S_{L_1}(T_-,T_+)\}
\end{displaymath}

\noindent is measurable. So we only need to check that $L_{T_-,T_+}(\Phi\ast_{CV_N})$ depends measurably on $(T_-,T_+)$ for all $\Phi\in\text{Out}(F_N)$. As $Erg_N$ is a Borel subset of $\mathbb{P}Curr_N$ by \cite[Proposition 1.3]{Phe66}, finiteness of $\text{Erg}(T)$ for all $T\in\mathcal{FI}$ (Proposition \ref{dual-arational}) and continuity of the intersection form imply the existence of countably many measurable maps $f_k:\mathcal{FI}\to \mathbb{P}Curr_N$ so that for all $T\in\mathcal{FI}$, we have $\text{Erg}(T)=\{f_k(T)|k\in\mathbb{N}\}$, see \cite{Cas67}. Using again continuity of the intersection form, this ensures that $L_{T_-,T_+}(\Phi\ast_{CV_N})$ depends measurably on $(T_-,T_+)$ (notice that for any positive pair $([\eta_-],[\eta_+])$ of currents, the supremum in the definition of $L_{[\eta_-],[\eta_+]}(\Phi\ast_{CV_N})$ can be taken on a dense countable subset of $\overline{CV_N}$).
\end{proof}

\section{The free factor complex}

We now give another interpretation of our results, by realizing the random walk on $\text{Out}(F_N)$ on the complex of free factors of $F_N$, instead of realizing it on $CV_N$. The \emph{free factor complex} $\mathcal{FF}_N$, introduced by Hatcher and Vogtmann in \cite{HV98}, is defined when $N\ge 3$ as the simplicial complex whose vertices are the conjugacy classes of nontrivial proper free factors of $F_N$, and higher dimensional simplices correspond to chains of inclusions of free factors. (When $N=2$, one has to modify this definition by adding an edge between any two complementary free factors to ensure that $\mathcal{FF}_2$ remains connected, and $\mathcal{FF}_2$ is isomorphic to the Farey graph). There is a natural, coarsely well-defined map $\psi:CV_N\to\mathcal{FF}_N$, that maps any tree $T\in CV_N$ to one of the conjugacy classes of the cyclic free factors of $F_N$ generated by an element of $F_N$ whose axis in $T$ projects to an embedded simple loop in the quotient graph $T/F_N$. When equipped with the simplicial metric, the free factor complex is Gromov hyperbolic  (\cite{BF12}, see also \cite{KR12}). Bestvina and Reynolds, and independently Hamenstädt, have determined its Gromov boundary $\partial\mathcal{FF}_N$.  

\begin{theo}\label{factor-boundary} (Bestvina-Reynolds \cite{BR13}, Hamenstädt \cite{Ham12})
There exists a unique $\text{Out}(F_N)$-equivariant homeomorphism $\partial\psi:\mathcal{AT}\to\partial \mathcal{FF}_N$, so that for all $T\in\mathcal{AT}$ and all sequences $(T_n)_{n\in\mathbb{N}}\in CV_N^{\mathbb{N}}$ that converge to $T$, the sequence $(\psi(T_n))_{n\in\mathbb{N}}$ converges to $\partial\psi(T)$.
\end{theo}

As a consequence of Theorems \ref{convergence-drift}, \ref{Poisson} and \ref{factor-boundary}, we therefore get the following result. The first part of the statement was obtained with different methods by Calegari and Maher \cite[Theorem 5.34]{CM13}, who worked in the more general context of isometry groups of (possibly nonproper) hyperbolic metric spaces.

\begin{theo}
Let $\mu$ be a probability measure on $\text{Out}(F_N)$, such that $gr(\mu)$ is nonelementary. Then for $\mathbb{P}$-almost every sample path $\mathbf{g}:=(g_n)_{n\in\mathbb{N}}$ of the random walk $(\text{Out}(F_N),\mu)$, there exists $\xi(\mathbf{g})\in\partial\mathcal{FF}_N$, such that for all $x\in\mathcal{FF}_N$, the sequence $(g_n x)_{n\in\mathbb{N}}$ converges to $\xi(\mathbf{g})$. The hitting measure $\nu$ on $\partial\mathcal{FF}_N$ is the unique $\mu$-stationary measure on $\partial\mathcal{FF}_N$. If in addition, the measure $\mu$ has finite first logarithmic moment with respect to the word metric on $\text{Out}(F_N)$, and finite entropy, then $(\partial\mathcal{FF}_N,\nu)$ is the Poisson boundary of $(\text{Out}(F_N),\mu)$. 
\qed
\end{theo}

\bibliographystyle{amsplain}
\bibliography{/Users/Camille/Documents/Bibliographie}

\end{document}